\documentclass[12pt]{amsart}
\usepackage{graphicx}
\usepackage{amsmath}
\usepackage{amssymb}
\usepackage{setspace}
\usepackage{amsthm}
\newtheorem{thm}{Theorem}[section]
\newtheorem{cor}[thm]{Corollary}

\newtheorem{lem}[thm]{Lemma}

\theoremstyle{definition}
\newtheorem{defn}[thm]{Definition}
\theoremstyle{remark}
\newtheorem{rem}[thm]{Remark}
\theoremstyle{conclusion}

\theoremstyle{conjecture}

\numberwithin{equation}{section}
\begin{document}
\title[Liouville type theorems in exterior domains]{Liouville type theorems for elliptic equations with Dirichlet conditions in exterior domains}

\author{Wei Dai, Guolin Qin}

\address{School of Mathematical Sciences, Beihang University (BUAA), Beijing 100083, P. R. China, and LAGA, UMR 7539, Institut Galil\'{e}e, Universit\'{e} Sorbonne Paris Nord, 93430 - Villetaneuse, France}
\email{weidai@buaa.edu.cn}

\address{Institute of Applied Mathematics, Chinese Academy of Sciences, Beijing 100190, and University of Chinese Academy of Sciences, Beijing 100049, P. R. China}
\email{qinguolin18@mails.ucas.ac.cn}

\thanks{Wei Dai is supported by the NNSF of China (No. 11971049 and 11501021), the Fundamental Research Funds for the Central Universities and the State Scholarship Fund of China (No. 201806025011).}

\begin{abstract}
In this paper, we are mainly concerned with the Dirichlet problems in exterior domains for the following elliptic equations:
\begin{equation}\label{GPDE0}
  (-\Delta)^{\frac{\alpha}{2}}u(x)=f(x,u) \,\,\,\,\,\,\,\,\,\,\,\, \text{in} \,\,\,\, \Omega_{r}:=\{x\in\mathbb{R}^{n}\,|\,|x|>r\}
\end{equation}
with arbitrary $r>0$, where $n\geq2$, $0<\alpha\leq 2$ and $f(x,u)$ satisfies some assumptions. A typical case is the Hardy-H\'{e}non type equations in exterior domains. We first derive the equivalence between \eqref{GPDE0} and the corresponding integral equations
\begin{equation}\label{GIE0}
	u(x)=\int_{\Omega_{r}}G_{\alpha}(x,y)f(y,u(y))dy,
\end{equation}
where $G_{\alpha}(x,y)$ denotes the Green's function for $(-\Delta)^{\frac{\alpha}{2}}$ in $\Omega_{r}$ with Dirichlet boundary conditions. Then, we establish Liouville theorems for \eqref{GIE0} via the method of scaling spheres developed in \cite{DQ0} by Dai and Qin, and hence obtain the Liouville theorems for \eqref{GPDE0}. Liouville theorems for integral equations related to higher order Navier problems in $\Omega_{r}$ are also derived.
\end{abstract}
\maketitle {\small {\bf Keywords:} The method of scaling spheres, Hardy-H\'{e}non type equations, Liouville theorems, nonnegative solutions, exterior domains. \\

}

\section{Introduction}

In this paper, we investigate the Liouville property of nonnegative solutions to the following Dirichlet problems for elliptic equations in exterior domains
\begin{equation}\label{GPDE}\\\begin{cases}
(-\Delta)^{\frac{\alpha}{2}}u(x)=f(x,u(x)), \,\,\,\,\,\,\, u(x)\geq0, \,\,\,\,\,\,\,\, x\in\Omega_{r}, \\
u(x)\equiv0,\,\,\,\,\,\,\,\, x\in\mathbb{R}^{n}\setminus\Omega_{r},
\end{cases}\end{equation}
where the exterior domains $\Omega_{r}:=\{x\in\mathbb{R}^{n}\,|\,|x|>r\}$ with arbitrary $r>0$, $n\geq2$, $0<\alpha\leq2$ and the nonlinear terms $f:\, \Omega_{r}\times\overline{\mathbb{R}_{+}}\rightarrow \overline{\mathbb{R}_{+}}$. When $0<\alpha<2$, the nonlocal fractional Laplacians $(-\Delta)^{\frac{\alpha}{2}}$ is defined by
\begin{equation}\label{nonlocal defn}
  (-\Delta)^{\frac{\alpha}{2}}u(x)=C_{\alpha,n} \, P.V.\int_{\mathbb{R}^n}\frac{u(x)-u(y)}{|x-y|^{n+\alpha}}dy:=C_{\alpha,n}\lim_{\epsilon\rightarrow0}\int_{|y-x|\geq\epsilon}\frac{u(x)-u(y)}{|x-y|^{n+\alpha}}dy
\end{equation}
for functions $u\in C^{1,1}_{loc}\cap\mathcal{L}_{\alpha}(\mathbb{R}^{n})$, where the constant $C_{\alpha,n}=\left(\int_{\mathbb{R}^{n}}\frac{1-\cos(2\pi\zeta_{1})}{|\zeta|^{n+\alpha}}d\zeta\right)^{-1}$ and the function spaces
\begin{equation}\label{0-1}
  \mathcal{L}_{\alpha}(\mathbb{R}^{n}):=\Big\{u: \mathbb{R}^{n}\rightarrow\mathbb{R}\,\Big|\,\int_{\mathbb{R}^{n}}\frac{|u(x)|}{1+|x|^{n+\alpha}}dx<\infty\Big\}.
\end{equation}
For $0<\alpha<2$, we assume the solution $u\in C_{loc}^{1,1}(\Omega_{r})\cap C(\overline{\Omega_{r}})\cap\mathcal{L}_{\alpha}(\mathbb{R}^{n})$. For $\alpha=2$, we assume the solution $u\in C^{2}(\Omega_{r})\cap C(\overline{\Omega_{r}})$.

We say equations \eqref{GPDE} have critical order if $\alpha=n$ and non-critical order if $0<\alpha<n$. The following definitions and assumptions on the nonlinear terms $f(x,u)$ will be necessary.
\begin{defn}\label{defn1}
In the non-critical order cases $0<\alpha<n$, we say that the nonlinear term $f$ has subcritical growth provided that
\begin{equation}\label{e1}
  \mu^{\frac{n+\alpha}{n-\alpha}}f(\mu^{\frac{2}{n-\alpha}}x,\mu^{-1}u)
\end{equation}
is strictly increasing with respect to $\mu\geq1$ or $\mu\leq1$ for all $(x,u)\in\Omega_{r}\times\mathbb{R}_{+}$. In the critical order cases $\alpha=n$, we say that the nonlinear term $f$ has subcritical growth, provided that
\begin{equation}\label{e1}
  \mu^{n}f(\mu x,u)
\end{equation}
is strictly increasing with respect to $\mu\geq1$ or $\mu\leq1$ for all $(x,u)\in\Omega_{r}\times\mathbb{R}_{+}$.
\end{defn}
\begin{defn}\label{defn2}
A function $g(x,u)$ is called locally Lipschitz on $u$ in $\Omega_{r}\times\overline{\mathbb{R}_{+}}$, provided that for any $u_{0}\in\overline{\mathbb{R}_{+}}$ and $\omega\subseteq\Omega_{r}$ bounded, there exists a (relatively) open neighborhood $U(u_{0})\subset\overline{\mathbb{R}_{+}}$ such that $g$ is Lipschitz continuous on $u$ in $\omega\times U(u_{0})$.
\end{defn}

We need the following three assumptions on the nonlinear terms $f(x,u)$. \\
$(\mathbf{f_{1}})$ The nonlinear term $f(x,u)$ is non-decreasing about $u$ in $\Omega_{r}\times\overline{\mathbb{R}_{+}}$, namely,
\begin{equation}\label{e2}
  (x,u), \, (x,v)\in\Omega_{r}\times\overline{\mathbb{R}_{+}} \,\,\, \text{with} \,\,\, u\leq v \,\,\, \text{implies} \,\,\, f(x,u)\leq f(x,v).
\end{equation}
$(\mathbf{f_{2}})$ There exists a $\theta<\frac{\alpha}{n}$ such that, $(|x|-r)^{\theta}f(x,u)$ is locally Lipschitz on $u$ in $\Omega_{r}\times\overline{\mathbb{R}_{+}}$. \\
$(\mathbf{f_{3}})$ There exist a cone $\mathcal{C}$ with vertex at $0$, constants $\overline{C}>0$, $\sigma>1$, $-\alpha<\tau<+\infty$ and $0<p<\frac{n+\alpha+2\tau}{n-\alpha}$ if $0<\alpha<n$, or $-\alpha\leq\tau<+\infty$ and $0<p<+\infty$ if $\alpha=n$ such that, the nonlinear term
\begin{equation}\label{e3}
  f(x,u)\geq \overline{C}|x|^{\tau}u^{p}
\end{equation}
in $(\mathcal{C}\cap\Omega_{2\sigma r})\times\overline{\mathbb{R}_{+}}$.

\begin{rem}\label{rem0}
In particular, assume $0\leq b<+\infty$, $-b-\alpha<a<+\infty$ and $1\leq p<\frac{n+\alpha+2(a+b)}{n-\alpha}$ if $0<\alpha<n$, or $-b-\alpha\leq a<+\infty$ and $1\leq p<+\infty$ if $\alpha=n$, then
\begin{equation}\label{0-2}
  f(x,u)=|x|^{a}(|x|-r)^{b}u^{p}
\end{equation}
is subcritical and satisfies all the assumptions $(\mathbf{f_{1}})$, $(\mathbf{f_{2}})$ and $(\mathbf{f_{3}})$. Moreover, under the same assumptions, for any $i=1,\cdots,n$, nonlinearities
\begin{equation}\label{0-3}
  f(x,u)=|x_{i}|^{a}(|x|-r)^{b}u^{p}, \quad |x|^{a}(|x_{i}|-r)^{b}u^{p} \quad \text{or} \quad |x_{i}|^{a}(|x_{i}|-r)^{b}u^{p}
\end{equation}
are also subcritical and satisfy all the assumptions $(\mathbf{f_{1}})$, $(\mathbf{f_{2}})$ and $(\mathbf{f_{3}})$.
\end{rem}

For $0<\alpha\leq n$, PDEs of the form
\begin{equation}\label{PDE}
  (-\Delta)^{\frac{\alpha}{2}}u(x)=|x|^{a}u^{p}(x)
\end{equation}
are called the fractional order or higher order H\'{e}non, Lane-Emden, Hardy equations for $a>0$, $a=0$, $a<0$, respectively. These equations have numerous important applications in conformal geometry and Sobolev inequalities. In particular, in the case $a=0$, \eqref{GPDE} becomes the well-known Lane-Emden equation, which models many phenomena in mathematical physics and in astrophysics.

The nonlinear terms in \eqref{PDE} is called critical if $p=p_{s}(a):=\frac{n+\alpha+2a}{n-\alpha}$ ($:=+\infty$ if $n=\alpha$), subcritical if $0<p<p_{s}(a)$ and supercritical if $p_{s}(a)<p<+\infty$. Liouville type theorems for equations \eqref{PDE} (i.e., nonexistence of nontrivial nonnegative solutions) in the whole space $\mathbb{R}^n$, the half space $\mathbb{R}^n_+$ and bounded domains $\Omega$ have been extensively studied (see \cite{BG,BP,CD,CDQ,CFY,CL,CLL,CLZ,DPQ,DQ1,DQ0,DQ,DQ2,DQZ,DZ,GS,Lei,Lin,MP,PS,RW,RZ,WX} and the references therein). For other related properties on PDEs \eqref{PDE} and Liouville type theorems on systems of PDEs of type \eqref{PDE} with respect to various types of solutions (e.g., stable, radial, singular, nonnegative, sign-changing, $\cdots$), please refer to \cite{BG,CD,CGS,CLO,DFQ,DQ1,DQ,FG,LB,Lei,Lin,M,PQS,WX} and the references therein. These Liouville theorems, in conjunction with the blowing up and re-scaling arguments, are crucial in establishing a priori estimates and hence existence of positive solutions to non-variational boundary value problems for a class of elliptic equations on bounded domains or on Riemannian manifolds with boundaries (see \cite{CDQ,DPQ,DQ0,DQ2,GS1,MP,PQS,RW}).

In this paper, by applying the method of scaling spheres developed in \cite{DQ0}, we will establish Liouville theorems for nonnegative solutions of the generalized equations \eqref{GPDE} with Dirichlet boundary conditions in unbounded exterior domains.

First, by using similar arguments as in \cite{CFY,ZCCY} (see also \cite{DFQ,DQ0}), we can deduce the equivalence between PDEs \eqref{GPDE} and the following integral equations
\begin{equation}\label{IEe}
  u(x)=\int_{\Omega_{r}}G_{\alpha}(x,y)f(y,u(y))dy,
\end{equation}
where
\begin{equation}\label{GREENe-0}
  G_{\alpha}(x,y):=\frac{C_{n,\alpha}}{|x-y|^{n-\alpha}}\int_{0}^{\frac{(|x|^{2}-r^{2})(|y|^{2}-r^{2})}{r^{2}|x-y|^{2}}}\frac{b^{\frac{\alpha}{2}-1}}{(1+b)^{\frac{n}{2}}}db \quad\quad \text{if} \,\,\, x,y\in \Omega_{r},
\end{equation}
and $G_{\alpha}(x,y):=0$ if $x$ or $y\in\mathbb{R}^{n}\setminus \Omega_{r}$ is the Green's function in exterior domain $\Omega_{r}$ for $(-\Delta)^{\frac{\alpha}{2}}$ with Dirichlet boundary conditions when $0<\alpha\leq2$ and $n\geq2$. That is, we have the following theorem.
\begin{thm}\label{equivalence}
If $u$ is a nonnegative solution of \eqref{GPDE}, then $u$ is also a nonnegative solution of integral equation \eqref{IEe}, and vice versa.
\end{thm}

\begin{rem}\label{rem2}
Theorem \ref{equivalence} can be proved through entirely similar arguments as in \cite{CFY,ZCCY}, so we omit the details here.
\end{rem}

Next, we consider the integral equations \eqref{IEe} instead of PDEs \eqref{GPDE}. We will study the integral equations \eqref{IEe} via the method of scaling spheres developed by Dai and Qin in \cite{DQ0}. The method of scaling spheres is essentially a frozen variant of the method of moving spheres, that is, we only dilate or shrink the spheres with respect to one fixed center. The method of moving spheres was initially used by Padilla \cite{Pa}, Chen and Li \cite{CL1}, and developed by Li and Zhu \cite{LZ} (two key calculus lemmas were established therein), which means moving spheres centered at every points in $\mathbb{R}^{n}$ or $\partial\mathbb{R}^{n}_{+}$ in conjunction with calculus lemmas and ODE analysis. Later, it was further developed by Li \cite{Li}, Chen and Li \cite{CL0}, Jin, Li and Xu \cite{JLX}. Recently, Chen, Li and Zhang developed a direct method of moving spheres on fractional order equations in \cite{CLZ}. One should note that, being different from the method of moving spheres, the method of scaling spheres take full advantage of the integral representation formulae of solutions and can be applied to various PDE or IE problems with singularities or without translation invariance on general domains. It can also be applied to various fractional or higher order problems in the cases that the method of moving planes in conjunction with Kelvin transforms do not work (see \cite{DQ0,DQ2,DQZ}). The method of scaling spheres, in conjunction with the integral representation formulae of solutions and a ``Bootstrap" iteration process, will provide useful lower bound estimates on the asymptotic behaviour of solutions, which will lead to a contradiction with the integrability of solutions unless the solution $u\equiv0$.

Our Liouville type result for IEs \eqref{IEe} is the following theorem.
\begin{thm}\label{Thm0}
Assume $n\geq2$, $0<\alpha\leq2$, $f$ is subcritical and satisfies the assumptions $(\mathbf{f_{1}})$, $(\mathbf{f_{2}})$ and $(\mathbf{f_{3}})$. If $u\in C(\overline{\Omega_{r}})$ is a nonnegative solution to IEs \eqref{IEe}, then $u\equiv 0$ in $\overline{\Omega_{r}}$.
\end{thm}

As a consequence of Theorem \ref{equivalence} and \ref{Thm0}, we obtain immediately the following Liouville type theorem on PDEs \eqref{GPDE}.
\begin{thm}\label{Thm1}
Assume $n\geq2$, $0<\alpha\leq2$, $f$ is subcritical and satisfies the assumptions $(\mathbf{f_{1}})$, $(\mathbf{f_{2}})$ and $(\mathbf{f_{3}})$. Suppose $u$ is a nonnegative solution of PDEs \eqref{GPDE}, then $u\equiv0$ in $\overline{\Omega_{r}}$.
\end{thm}

\begin{rem}\label{rem3}
For $\alpha=2$ and $n\geq3$, Reichel and Zou \cite{RZ} have obtained some Liouville type theorems for equations \eqref{GPDE} under some assumptions. Theorem \ref{Thm1} improved the results in \cite{RZ} at least in three aspects. First, Reichel and Zou \cite{RZ} required that there exist some $\tau>-2$ and $1<p<\frac{n+\alpha+2\tau}{n-\alpha}$ such that $f(x,u)$ satisfies the lower bound \eqref{e3} in assumption $(\mathbf{f_{3}})$ on the whole $\Omega_{r}\times\overline{\mathbb{R}_{+}}$. But we only need to assume in assumption $(\mathbf{f_{3}})$ that there exist $-\alpha<\tau<+\infty$, $0<p<\frac{n+\alpha+2\tau}{n-\alpha}$ if $0<\alpha<n$ ($-\alpha\leq\tau<+\infty$, $0<p<+\infty$ if $\alpha=n$), a cone $\mathcal{C}$ with vertex at $0$ and $\sigma>1$ such that $f(x,u)$ satisfies the lower bound \eqref{e3} in $(\mathcal{C}\cap\Omega_{2\sigma r})\times\overline{\mathbb{R}_{+}}$. This allows us to have much more admissible choices of the nonlinearities $f(x,u)$ (see Remark \ref{rem0} and \ref{rem1}). Second, Theorem \ref{Thm1} can also be applied to general fractional order cases $0<\alpha<2$ with $n\geq2$ and the critical order cases $\alpha=n=2$. Third, in assumption $(\mathbf{f_{2}})$, we only assume $(|x|-r)^{\theta}f(x,u)$ (not $f(x,u)$ itself) is locally Lipschitz on $u$, this allows $f(x,u)$ to have some singularities near the sphere $S_{r}:=\{x\in\mathbb{R}^{n}\,|\,|x|=r\}$.
\end{rem}

In particular, we consider the following Dirichlet problems for the H\'{e}non-Hardy type equations in exterior domains
\begin{equation}\label{HPDE}\\\begin{cases}
(-\Delta)^{\frac{\alpha}{2}}u(x)=|x|^{a}(|x|-r)^{b}u^{p}, \,\,\,\,\,\,\, u(x)\geq0, \,\,\,\,\,\,\,\, x\in\Omega_{r}, \\
u(x)\equiv0,\,\,\,\,\,\,\,\, x\in\mathbb{R}^{n}\setminus\Omega_{r},
\end{cases}\end{equation}
where $n\geq2$, $0<\alpha\leq2$.

As a consequence of Theorem \ref{Thm1} and Remark \ref{rem0}, we deduce the following corollary.
\begin{cor}\label{Cor1}
Assume $0\leq b<+\infty$, $-b-\alpha<a<+\infty$ and $1\leq p<\frac{n+\alpha+2(a+b)}{n-\alpha}$ if $0<\alpha<n$, $-b-\alpha\leq a<+\infty$ and $1\leq p<+\infty$ if $\alpha=n=2$. Suppose $u$ is a nonnegative solution of \eqref{HPDE}, then $u\equiv0$ in $\overline{\Omega_{r}}$.
\end{cor}

\begin{rem}\label{rem1}
By Remark \ref{rem0}, if we suppose the nonlinearities in \eqref{GPDE} take the form:
\begin{equation}\label{0-3}
  f(x,u)=|x_{i}|^{a}(|x|-r)^{b}u^{p}, \quad |x|^{a}(|x_{i}|-r)^{b}u^{p} \quad \text{or} \quad |x_{i}|^{a}(|x_{i}|-r)^{b}u^{p}
\end{equation}
for any $i=1,\cdots,n$, then under the same assumptions as in Corollary \ref{Cor1}, the Liouville type results in Theorem \ref{Thm1} are also valid for equations \eqref{GPDE}.
\end{rem}

We also consider the following higher order integral equations
\begin{equation}\label{HIEe}
  u(x)=\int_{\Omega_{r}}G_{\alpha}(x,y)f(y,u(y))dy,
\end{equation}
where for $\alpha=2m$ with $1\leq m<\frac{n}{2}$ and $n\geq3$,
\begin{equation}\label{GREENe-H1}
  G_{\alpha}(x,y):=C_{n,\alpha}\Bigg(\frac{1}{|x-y|^{n-\alpha}}-\frac{1}{\left|\frac{rx}{|x|}-\frac{|x|y}{r}\right|^{n-\alpha}}\Bigg) \qquad \text{if} \,\,\, x,y\in \Omega_{r},
\end{equation}
while for $\alpha=n$ with $n\geq2$ even,
\begin{equation}\label{GREENe-H2}
  G_{\alpha}(x,y):=C_{n,\alpha}\Bigg(\ln\frac{1}{|x-y|}-\ln\frac{1}{\left|\frac{rx}{|x|}-\frac{|x|y}{r}\right|}\Bigg) \qquad \text{if} \,\,\, x,y\in \Omega_{r},
\end{equation}
and $G_{\alpha}(x,y):=0$ if $x$ or $y\in\mathbb{R}^{n}\setminus \Omega_{r}$ is the Green's function in exterior domain $\Omega_{r}$ for $(-\Delta)^{\frac{\alpha}{2}}$ with Navier boundary conditions when $\alpha=2m$ with $1\leq m\leq\frac{n}{2}$ and $n\geq2$. The integral equations \eqref{HIEe} is closely related to higher order Navier problems in $\Omega_{r}$.

By entirely similar arguments as in the proof of Theorem \ref{Thm0}, we can prove the following Liouville theorem for integral equations \eqref{HIEe}.
\begin{thm}\label{Thm0-H}
Assume $n\geq2$, $\alpha=2m$ with $1\leq m\leq\frac{n}{2}$, $f$ is subcritical and satisfies the assumptions $(\mathbf{f_{1}})$, $(\mathbf{f_{2}})$ and $(\mathbf{f_{3}})$. If $u\in C(\overline{\Omega_{r}})$ is a nonnegative solution to IEs \eqref{HIEe}, then $u\equiv 0$ in $\overline{\Omega_{r}}$.
\end{thm}

\begin{rem}\label{rem5}
Theorem \ref{Thm0-H} can be proved through entirely similar arguments as Theorem \ref{Thm0}, so we omit the details here.
\end{rem}

\begin{rem}\label{rem4}
Consider the following Navier problems for higher order elliptic equations in $\Omega_{r}$:
\begin{equation}\label{NPDE-H}\\\begin{cases}
(-\Delta)^{\frac{\alpha}{2}} u(x)=f(x,u(x)), \,\,\,\,\,\,\,\, u(x)\geq0, \,\,\,\,\, \,\,\, x\in\Omega_{r}, \\
u=(-\Delta)u=\cdots=(-\Delta)^{\frac{\alpha}{2}-1}u=0 \,\,\,\,\,\,\,\,\,\,\, \text{on} \,\,\,  S_{r},
\end{cases}\end{equation}
where $u\in C^{\alpha}(\Omega_{r})\cap C^{\alpha-2}(\overline{\Omega_{r}})$, $n\geq2$, $\alpha=2m$ with $1\leq m\leq\frac{n}{2}$, $f$ is subcritical and satisfies the assumptions $(\mathbf{f_{1}})$, $(\mathbf{f_{2}})$ and $(\mathbf{f_{3}})$. Once the equivalence between the Navier problems \eqref{NPDE-H} and the integral equations \eqref{HIEe} has been established, we can also derive from Theorem \ref{Thm0-H} that Liouville theorem for nonnegative solutions to the Navier problems \eqref{NPDE-H} holds.
\end{rem}

In the following, we will use $C$ to denote a general positive constant that may depend on $n$, $\alpha$, $\tau$, $\sigma$, $\theta$, $p$, $\overline{C}$, $u$ and the cone $\mathcal{C}$, and whose value may differ from line to line.

\section{Proof of Theorem \ref{Thm0}}

In this section, we will prove Theorem \ref{Thm0} via contradiction arguments and the method of scaling spheres. Without loss of generality, we may assume the radius $r=1$. We may also assume that the nonlinear term $f(x,u)$ satisfies subcritical conditions in Definition \ref{defn1} for $\mu\leq1$. If $f(x,u)$ satisfies subcritical conditions in Definition \ref{defn1} for $\mu\geq1$, we only need to carry out calculations and estimates inside the ball $B_{\lambda}(0)$ during the scaling spheres procedure.

Now suppose on the contrary that $u\geq0$ satisfies the equivalent integral equations \eqref{IEe} but $u$ is not identically zero, then one can infer from the integral equations \eqref{IEe} that $u$ is actually a positive solution, i.e., $u>0$ in $\Omega_{1}$. Next, we will carry out our proof by discussing the non-critical order cases and the critical order case separately.

\subsection{The non-critical order cases $0<\alpha<n$}

We will apply the method of scaling spheres to show the following lower bound estimates for asymptotic behaviour of positive solution $u$ as $|x|\rightarrow+\infty$, which will contradict with the integral equations \eqref{IEe}.
\begin{thm}\label{lower-e}
Assume $n\geq2$, $n>\alpha$, $0<\alpha\leq2$, $f(x,u)$ is subcritical and satisfies assumptions $(\mathbf{f_{1}})$, $(\mathbf{f_{2}})$ and $(\mathbf{f_{3}})$. Suppose $u$ is a positive solution to integral equations \eqref{IEe}, then it satisfies the following lower bound estimates: for all $|x|\geq2\sigma$,
\begin{equation}\label{lb1-e}
  u(x)\geq C_{\kappa}|x|^{\kappa} \quad\quad \forall \, \kappa<\frac{\alpha+\tau}{1-p}, \quad\quad \text{if} \,\,\,\, 0<p<1;
\end{equation}
\begin{equation}\label{lb2-e}
  u(x)\geq C_{\kappa}|x|^{\kappa} \quad\quad \forall \, \kappa<+\infty, \quad\quad \text{if} \,\,\,\, 1\leq p<\frac{n+\alpha+2\tau}{n-\alpha}.
\end{equation}
\end{thm}
\begin{proof}
Given any $\lambda>1$, we define the Kelvin transform of $u$ centered at $0$ by
\begin{equation}\label{Kelvin-e}
  u_{\lambda}(x):=\left(\frac{\lambda}{|x|}\right)^{n-\alpha}u\left(\frac{\lambda^{2}x}{|x|^{2}}\right)
\end{equation}
for arbitrary $x\in\{x\in\overline{\Omega_{1}}\,|\,1\leq|x|\leq\lambda^{2}\}$, and define the reflection of $x$ about the sphere $S_{\lambda}:=\{x\in\mathbb{R}^{n}\,|\,|x|=\lambda\}$ by $x^{\lambda}:=\frac{\lambda^{2}x}{|x|^{2}}$.

Now, we will carry out the process of scaling spheres in $\Omega_{1}$ with respect to the origin $0\in\mathbb{R}^{n}$.

Let $\lambda>1$ be an arbitrary real number and let $\omega^{\lambda}(x):=u_{\lambda}(x)-u(x)$ for any $x\in B_{\lambda^{2}}(0)\setminus\overline{B_{\lambda}(0)}$. We will first show that, for $\lambda>1$ sufficiently close to $1$,
\begin{equation}\label{2-7-e}
  \omega^{\lambda}(x)\leq0, \,\,\,\,\,\, \forall \,\, x\in B_{\lambda^{2}}(0)\setminus\overline{B_{\lambda}(0)}.
\end{equation}
Then, we start dilating the sphere $S_{\lambda}$ from near the unit sphere $S_{1}$ outward as long as \eqref{2-7-e} holds, until its limiting position $\lambda=+\infty$ and derive lower bound estimates on asymptotic behaviour of $u$ as $|x|\rightarrow+\infty$. Therefore, the scaling sphere process can be divided into two steps.

\emph{Step 1. Start dilating the sphere $S_{\lambda}$ from near $\lambda=1$.} Define
\begin{equation}\label{2-8-e}
  (B_{\lambda^{2}}\setminus\overline{B_{\lambda}})^{+}:=\{x\in B_{\lambda^{2}}(0)\setminus\overline{B_{\lambda}(0)} \, | \, \omega^{\lambda}(x)>0\}.
\end{equation}
We will show that, for $\lambda>1$ sufficiently close to $1$,
\begin{equation}\label{2-9-e}
  (B_{\lambda^{2}}\setminus\overline{B_{\lambda}})^{+}=\emptyset.
\end{equation}

Since $u$ is a positive solution to the integral equations \eqref{IEe}, through direct calculations, we get, for any $\lambda>1$,
\begin{equation}\label{2-34-e}
  u(x)=\int_{|y|>\lambda}G_{\alpha}(x,y)f(y,u(y))dy+\int_{B_{\lambda^{2}}(0)\setminus\overline{B_{\lambda}(0)}}G_{\alpha}(x,y^{\lambda})
  \left(\frac{\lambda}{|y|}\right)^{2n}f(y^{\lambda},u(y^{\lambda}))dy
\end{equation}
for any $x\in\overline{\Omega_{1}}$. By direct calculations, one can also verify that $u_{\lambda}$ satisfies the following integral equation
\begin{equation}\label{2-35-e}
  u_{\lambda}(x)=\int_{|y|>1}G_{\alpha}(x^{\lambda},y)\left(\frac{\lambda}{|x|}\right)^{n-\alpha}f(y,u(y))dy
\end{equation}
for any $x\in\{x\in\overline{\Omega_{1}}\,|\,1\leq|x|\leq\lambda^{2}\}$, and hence, it follows immediately that
\begin{eqnarray}\label{2-36-e}
  u_{\lambda}(x)&=&\int_{|y|>\lambda}G_{\alpha}(x^{\lambda},y)\left(\frac{\lambda}{|x|}\right)^{n-\alpha}f(y,u(y))dy \\
 \nonumber \quad\quad &&+\int_{B_{\lambda^{2}}(0)\setminus\overline{B_{\lambda}(0)}}G_{\alpha}(x^{\lambda},y^{\lambda})\left(\frac{\lambda}{|x|}\right)^{n-\alpha}
 \left(\frac{\lambda}{|y|}\right)^{2n}f(y^{\lambda},u(y^{\lambda}))dy.
\end{eqnarray}
Therefore, we have, for any $x\in B_{\lambda^{2}}(0)\setminus\overline{B_{1}(0)}$,
\begin{eqnarray}\label{omega-e}
% \nonumber to remove numbering (before each equation)
  && \omega^{\lambda}(x)=u_{\lambda}(x)-u(x) \\
  \nonumber &=& \int_{B_{\lambda^{2}}(0)\setminus\overline{B_{\lambda}(0)}}\Bigg\{\left[\left(\frac{\lambda}{|x|}\right)^{n-\alpha}G_{\alpha}(x^{\lambda},y^{\lambda})
  -G_{\alpha}(x,y^{\lambda})\right]\left(\frac{\lambda}{|y|}\right)^{2n}f(y^{\lambda},u(y^{\lambda})) \\
 \nonumber && -\left[G_{\alpha}(x,y)-\left(\frac{\lambda}{|x|}\right)^{n-\alpha}G_{\alpha}(x^{\lambda},y)\right]f(y,u(y))\Bigg\}dy \\
 \nonumber && +\int_{|y|>\lambda^{2}}\left[\left(\frac{\lambda}{|x|}\right)^{n-\alpha}G_{\alpha}(x^{\lambda},y)-G_{\alpha}(x,y)\right]f(y,u(y))dy.
\end{eqnarray}

Now we need the following Lemma on properties of the Green's function $G_{\alpha}(x,y)$.
\begin{lem}\label{G-e}
The Green's function $G_{\alpha}(x,y)$ satisfies the following point-wise estimates:
\begin{flalign}
\nonumber &\text{$(i)\,\, 0\leq G_{\alpha}(x,y)\leq\frac{C'}{|x-y|^{n-\alpha}}, \quad\quad \forall \,\, x,y\in\mathbb{R}^{n};$}& \\
\nonumber &\text{$(ii) \,\, G_{\alpha}(x,y)\geq\frac{C''}{|x-y|^{n-\alpha}}, \quad\quad \forall \,\, |x|,|y|\geq2;$}& \\
\nonumber &\text{$(iii) \,\, \left(\frac{\lambda}{|x|}\right)^{n-\alpha}G_{\alpha}(x^{\lambda},y)-G_{\alpha}(x,y)\leq0, \quad\quad \forall \,\, \lambda<|x|<\lambda^{2}, \,\, \lambda<|y|<+\infty;$}& \\
\nonumber &\text{$(iv) \,\, \left(\frac{\lambda^{2}}{|x|\cdot|y|}\right)^{n-\alpha}G_{\alpha}(x^{\lambda},y^{\lambda})-\left(\frac{\lambda}{|y|}\right)^{n-\alpha}
  G_{\alpha}(x,y^{\lambda})\leq G_{\alpha}(x,y)-\left(\frac{\lambda}{|x|}\right)^{n-\alpha}G_{\alpha}(x^{\lambda},y),$}& \\
\nonumber &\text{$\qquad\qquad\qquad\qquad\qquad\qquad\qquad\qquad\qquad\qquad\qquad\qquad\qquad\qquad\quad \forall \,\, \lambda<|x|,|y|<\lambda^{2}.$}&
\end{flalign}
\end{lem}

Lemma \ref{G-e} can be proved by direct calculations, so we omit the details here.

From Lemma \ref{G-e} and the integral equations \eqref{omega-e}, one can derive that, for any $x\in B_{\lambda^{2}}(0)\setminus\overline{B_{\lambda}(0)}$,
\begin{eqnarray}\label{2-37-e}
% \nonumber to remove numbering (before each equation)
  &&\omega^{\lambda}(x)=u_{\lambda}(x)-u(x) \\
 \nonumber &\leq&\int_{\lambda<|y|<\lambda^{2}}\Bigg[G_{\alpha}(x,y)-\left(\frac{\lambda}{|x|}\right)^{n-\alpha}G_{\alpha}(x^{\lambda},y)\Bigg] \left[\left(\frac{\lambda}{|y|}\right)^{n+\alpha}f(y^{\lambda},u(y^{\lambda}))-f(y,u(y))\right]dy\\
\nonumber &<&\int_{\lambda<|y|<\lambda^{2}}\Bigg(G_{\alpha}(x,y)-\left(\frac{\lambda}{|x|}\right)^{n-\alpha}G_{\alpha}(x^{\lambda},y)\Bigg) \left[f(y,u_{\lambda}(y))-f(y,u(y))\right]dy\\
\nonumber &\leq&C\int_{\left(B_{\lambda^{2}}\setminus\overline{B_{\lambda}}\right)^{+}}\frac{1}{|x-y|^{n-\alpha}}\left[f(y,u_{\lambda}(y))-f(y,u(y))\right]dy\\
\nonumber &=&C\int_{\left(B_{\lambda^{2}}\setminus\overline{B_{\lambda}}\right)^{+}}\frac{1}{|x-y|^{n-\alpha}}\cdot\frac{f(y,u_{\lambda}(y))-f(y,u(y))}{u_{\lambda}(y)-u(y)}
\omega^{\lambda}(y)dy,
\end{eqnarray}
where we have used the subcritical condition on $f(x,u)$ for $\mu=\left(\frac{\lambda}{|y|}\right)^{n-\alpha}<1$ to derive the second inequality and the assumption $(\mathbf{f_{1}})$ on $f(x,u)$ to derive the third inequality.

By Hardy-Littlewood-Sobolev inequality and \eqref{2-37-e}, we have, for any $\frac{n}{n-\alpha}<q<\infty$,
\begin{eqnarray}\label{3-14-e}
 && \|\omega^{\lambda}\|_{L^{q}((B_{\lambda^{2}}\setminus\overline{B_{\lambda}})^{+})}\leq C\left\|\frac{f(y,u_{\lambda}(y))-f(y,u(y))}{u_{\lambda}(y)-u(y)}
\omega^{\lambda}(y)\right\|_{L^{\frac{nq}{n+\alpha q}}((B_{\lambda^{2}}\setminus\overline{B_{\lambda}})^{+})}\\
  \nonumber &\leq& C\left\|\frac{f(y,u_{\lambda}(y))-f(y,u(y))}{u_{\lambda}(y)-u(y)}\right\|_{L^{\frac{n}{\alpha}}((B_{\lambda^{2}}\setminus\overline{B_{\lambda}})^{+})}
  \cdot\|\omega^{\lambda}\|_{L^{q}((B_{\lambda^{2}}\setminus\overline{B_{\lambda}})^{+})}.
\end{eqnarray}
Since $u\in C(\overline{\Omega_{1}})$ and $f(x,u)$ satisfies the assumption $(\mathbf{f_{2}})$, there exists a $\epsilon_{0}>0$ small enough, such that
\begin{equation}\label{3-15-e}
  C\left\|\frac{f(y,u_{\lambda}(y))-f(y,u(y))}{u_{\lambda}(y)-u(y)}\right\|_{L^{\frac{n}{\alpha}}((B_{\lambda^{2}}\setminus\overline{B_{\lambda}})^{+})}\leq\frac{1}{2}
\end{equation}
for all $1<\lambda\leq1+\epsilon_{0}$, and hence \eqref{3-14-e} implies
\begin{equation}\label{3-16-e}
  \|\omega^{\lambda}\|_{L^{q}((B_{\lambda^{2}}\setminus\overline{B_{\lambda}})^{+})}=0,
\end{equation}
which means $(B_{\lambda^{2}}\setminus\overline{B_{\lambda}})^{+}=\emptyset$. Therefore, we have proved for all $1<\lambda\leq1+\epsilon_{0}$, $(B_{\lambda^{2}}\setminus\overline{B_{\lambda}})^{+}=\emptyset$, that is,
\begin{equation}\label{3-17-e}
  \omega^{\lambda}(x)\leq0, \,\,\,\,\,\,\, \forall \, x\in B_{\lambda^{2}}(0)\setminus\overline{B_{\lambda}(0)}.
\end{equation}
This completes Step 1.

\emph{Step 2. Dilate the sphere $S_{\lambda}$ outward until $\lambda=+\infty$ to derive lower bound estimates on asymptotic behaviour of $u$ as $|x|\rightarrow+\infty$.} Step 1 provides us a start point to dilate the sphere $S_{\lambda}$ from near $\lambda=1$. Now we dilate the sphere $S_{\lambda}$ outward as long as \eqref{2-7-e} holds. Let
\begin{equation}\label{2-29-e}
  \lambda_{0}:=\sup\{1<\lambda<+\infty\,|\, \omega^{\mu}\leq0 \,\, in \,\, B_{\mu^{2}}(0)\setminus\overline{B_{\mu}(0)}, \,\, \forall \, 1<\mu\leq\lambda\}\in(1,+\infty],
\end{equation}
and hence, one has
\begin{equation}\label{2-30-e}
  \omega^{\lambda_{0}}(x)\leq0, \quad\quad \forall \,\, x\in B_{\lambda_{0}^{2}}(0)\setminus\overline{B_{\lambda_{0}}(0)}.
\end{equation}
In what follows, we will prove $\lambda_{0}=+\infty$ by contradiction arguments.

Suppose on contrary that $1<\lambda_{0}<+\infty$. In order to get a contradiction, we will first prove
\begin{equation}\label{2-31-e}
  \omega^{\lambda_{0}}(x)\equiv0, \,\,\,\,\,\,\forall \, x\in B_{\lambda_{0}^{2}}(0)\setminus\overline{B_{\lambda_{0}}(0)}
\end{equation}
by using contradiction arguments.

Suppose on contrary that \eqref{2-31-e} does not hold, that is, $\omega^{\lambda_{0}}\leq0$ but $\omega^{\lambda_{0}}$ is not identically zero in $B_{\lambda_{0}^{2}}(0)\setminus\overline{B_{\lambda_{0}}(0)}$, then there exists a $x^{0}\in B_{\lambda_{0}^{2}}(0)\setminus\overline{B_{\lambda_{0}}(0)}$ such that $\omega^{\lambda_{0}}(x^{0})<0$. We will obtain a contradiction with \eqref{2-29-e} via showing that the sphere $S_{\lambda}$ can be dilated outward a little bit further, more precisely, there exists a $\varepsilon>0$ small enough such that $\omega^{\lambda}\leq0$ in $B_{\lambda^{2}}(0)\setminus\overline{B_{\lambda}(0)}$ for all $\lambda\in[\lambda_{0},\lambda_{0}+\varepsilon]$.

For that purpose, we will first show that
\begin{equation}\label{2-32-e}
  \omega^{\lambda_{0}}(x)<0, \,\,\,\,\,\, \forall \, x\in B_{\lambda_{0}^{2}}(0)\setminus\overline{B_{\lambda_{0}}(0)}.
\end{equation}
Indeed, since we have assumed there exists a point $x^{0}\in B_{\lambda_{0}^{2}}(0)\setminus\overline{B_{\lambda_{0}}(0)}$ such that $\omega^{\lambda_{0}}(x^{0})<0$, by continuity, there exists a small $\delta>0$ and a constant $c_{0}>0$ such that
\begin{equation}\label{2-33-e}
B_{\delta}(x^{0})\subset B_{\lambda_{0}^{2}}(0)\setminus\overline{B_{\lambda_{0}}(0)} \,\,\,\,\,\, \text{and} \,\,\,\,\,\,
\omega^{\lambda_{0}}(x)\leq -c_{0}<0, \,\,\,\,\,\,\,\, \forall \, x\in B_{\delta}(x^{0}).
\end{equation}
Since $f(x,u)$ is subcritical and satisfies the assumption $(\mathbf{f_{1}})$, one can derive from \eqref{2-33-e}, Lemma \ref{G-e} and \eqref{2-37-e} that, for any $x\in B_{\lambda_{0}^{2}}(0)\setminus\overline{B_{\lambda_{0}}(0)}$,
\begin{eqnarray}\label{9-37-e}
% \nonumber to remove numbering (before each equation)
  &&\omega^{\lambda_{0}}(x)=u_{\lambda_{0}}(x)-u(x) \\
  \nonumber &\leq&\int_{\lambda_{0}<|y|<\lambda_{0}^{2}}\Bigg[G_{\alpha}(x,y)-\left(\frac{\lambda_{0}}{|x|}\right)^{n-\alpha}G_{\alpha}(x^{\lambda_{0}},y)\Bigg] \left[\left(\frac{\lambda_{0}}{|y|}\right)^{n+\alpha}f(y^{\lambda_{0}},u(y^{\lambda_{0}}))-f(y,u(y))\right]dy\\
\nonumber &<&\int_{B_{\delta}(x_{0})}\Bigg(G_{\alpha}(x,y)-\left(\frac{\lambda_{0}}{|x|}\right)^{n-\alpha}G_{\alpha}(x^{\lambda_{0}},y)\Bigg) \left[f(y,u_{\lambda_{0}}(y))-f(y,u(y))\right]dy\leq0,
\end{eqnarray}
thus we arrive at \eqref{2-32-e}.

Now, we choose a $0<r_{0}<\frac{1}{4}\min\{\lambda_{0}^{2}-\lambda_{0},1\}$ small enough, such that
\begin{equation}\label{9-0e}
  C\left\|\frac{f(y,u_{\lambda}(y))-f(y,u(y))}{u_{\lambda}(y)-u(y)}\right\|_{L^{\frac{n}{\alpha}}\big(A_{\lambda_{0}+r_{0},r_{0}}\cup A_{\lambda_{0}^{2}+r_{0},2r_{0}}\big)}\leq\frac{1}{2}
\end{equation}
for any $\lambda\in[\lambda_{0},\lambda_{0}+\frac{r_{0}}{4}]$, where the constant $C$ is the same as in \eqref{3-15-e} and the narrow region
\begin{equation}\label{9-1e}
  A_{r,l}:=\left\{x\in B_{r}(0)\,\big|\,|x|>r-l\right\}
\end{equation}
for $r>0$ and $0<l<r$. By \eqref{2-37-e}, one can easily verify that inequality as \eqref{3-14-e} (with the same constant $C$) also holds for any $\lambda\in[\lambda_{0},\lambda_{0}+\frac{r_{0}}{4}]$, that is, for any $\frac{n}{n-\alpha}<q<\infty$,
\begin{equation}\label{9-2e}
  \|\omega^{\lambda}\|_{L^{q}\left((B_{\lambda^{2}}\setminus\overline{B_{\lambda}})^{+}\right)}\leq C\left\|\frac{f(y,u_{\lambda}(y))-f(y,u(y))}{u_{\lambda}(y)-u(y)}\right\|_{L^{\frac{n}{\alpha}}\left((B_{\lambda^{2}}\setminus\overline{B_{\lambda}})^{+}\right)}
  \cdot\|\omega^{\lambda}\|_{L^{q}\left((B_{\lambda^{2}}\setminus\overline{B_{\lambda}})^{+}\right)}.
\end{equation}

By \eqref{2-32-e}, we can define
\begin{equation}\label{2-40-e}
  M_{0}:=\sup_{x\in \overline{B_{\lambda_{0}^{2}-r_{0}}(0)}\setminus B_{\lambda_{0}+r_{0}}(0)}\omega^{\lambda_{0}}(x)<0.
\end{equation}
Since $u$ is uniformly continuous on arbitrary compact set $K\subset\overline{\Omega_{1}}$ (say, $K=\{x\in\overline{\Omega_{1}}\,|\,\lambda_{0}+r_{0}\leq|x|\leq2(\lambda_{0}^{2}-r_{0})\}$), we can deduce from \eqref{2-40-e} that, there exists a $0<\varepsilon_{0}<\frac{r_{0}}{4}$ sufficiently small, such that, for any $\lambda\in[\lambda_{0},\lambda_{0}+\varepsilon_{0}]$,
\begin{equation}\label{2-41-e}
  \omega^{\lambda}(x)\leq\frac{M_{0}}{2}<0, \,\,\,\,\,\, \forall \, x\in \overline{B_{\lambda_{0}^{2}-r_{0}}(0)}\setminus B_{\lambda_{0}+r_{0}}(0).
\end{equation}

For any $\lambda\in[\lambda_{0},\lambda_{0}+\varepsilon_{0}]$, it follows from \eqref{2-41-e} that
\begin{equation}\label{9-4e}
  (B_{\lambda^{2}}\setminus\overline{B_{\lambda}})^{+}\subset A_{\lambda_{0}+r_{0},r_{0}}\cup A_{\lambda_{0}^{2}+r_{0},2r_{0}}.
\end{equation}
As a consequence of \eqref{9-0e}, \eqref{9-2e} and \eqref{9-4e}, we get
\begin{equation}\label{9-5e}
  \|\omega^{\lambda}\|_{L^{q}\left((B_{\lambda^{2}}\setminus\overline{B_{\lambda}})^{+}\right)}=0,
\end{equation}
and hence $(B_{\lambda^{2}}\setminus\overline{B_{\lambda}})^{+}=\emptyset$ for all $\lambda\in[\lambda_{0},\lambda_{0}+\varepsilon_{0}]$, that is,
\begin{equation}\label{2-45-e}
  \omega^{\lambda}(x)\leq0, \,\,\,\,\,\,\, \forall \,\, x\in B_{\lambda^{2}}(0)\setminus\overline{B_{\lambda}(0)},
\end{equation}
which contradicts with the definition \eqref{2-29-e} of $\lambda_{0}$. As a consequence, in the case $1<\lambda_{0}<+\infty$, \eqref{2-31-e} must hold true, that is,
\begin{equation}\label{2-46-e}
  \omega^{\lambda_{0}}\equiv0 \,\,\,\,\,\, \text{in} \,\,\, B_{\lambda_{0}^{2}}(0)\setminus\overline{B_{\lambda_{0}}(0)}.
\end{equation}

However, by subcritical condition on $f(x,u)$, the first inequality in \eqref{9-37-e} and \eqref{2-46-e}, we arrive at
\begin{eqnarray}\label{2-47-e}
 && 0=\omega^{\lambda_{0}}(x)=u_{\lambda_{0}}(x)-u(x)\\
 \nonumber &\leq&\int_{\lambda_{0}<|y|<\lambda_{0}^{2}}\Bigg[G_{\alpha}(x,y)-\left(\frac{\lambda_{0}}{|x|}\right)^{n-\alpha}G_{\alpha}(x^{\lambda_{0}},y)\Bigg] \left[\left(\frac{\lambda_{0}}{|y|}\right)^{n+\alpha}f(y^{\lambda_{0}},u(y^{\lambda_{0}}))-f(y,u(y))\right]dy\\
\nonumber &<&\int_{\lambda_{0}<|y|<\lambda_{0}^{2}}\Bigg(G_{\alpha}(x,y)-\left(\frac{\lambda_{0}}{|x|}\right)^{n-\alpha}G_{\alpha}(x^{\lambda_{0}},y)\Bigg) \left[f(y,u_{\lambda_{0}}(y))-f(y,u(y))\right]dy=0
\end{eqnarray}
for any $x\in B_{\lambda_{0}^{2}}(0)\setminus\overline{B_{\lambda_{0}}(0)}$, which is absurd. Thus we must have $\lambda_{0}=+\infty$, that is,
\begin{equation}\label{2-48-e}
  u(x)\geq\left(\frac{\lambda}{|x|}\right)^{n-\alpha}u\left(\frac{\lambda^{2}x}{|x|^{2}}\right), \quad\quad \forall \,\, \lambda\leq|x|\leq\lambda^{2}, \quad \forall \,\, 1<\lambda<+\infty.
\end{equation}
For arbitrary $\sigma\leq|x|<+\infty$, let $1<\sqrt{|x|}<\lambda:=\sqrt{\sigma|x|}\leq|x|<+\infty$, then \eqref{2-48-e} yields that
\begin{equation}\label{2-49-e}
  u(x)\geq\left(\frac{\sigma}{|x|}\right)^{\frac{n-\alpha}{2}}u\left(\frac{\sigma x}{|x|}\right),
\end{equation}
and hence, we arrive at the following lower bound estimate on asymptotic behaviour of $u$ as $|x|\rightarrow+\infty$:
\begin{equation}\label{2-50-e}
  u(x)\geq\left(\min_{x\in S_{\sigma}}u(x)\right)\left(\frac{\sigma}{|x|}\right)^{\frac{n-\alpha}{2}}:=\frac{C_{0}}{|x|^{\frac{n-\alpha}{2}}}, \quad\quad \forall \,\, \sigma\leq|x|<\infty.
\end{equation}

The lower bound estimate \eqref{2-50-e} can be improved remarkably by using the ``Bootstrap" iteration technique and the integral equations \eqref{IEe}.

In fact, let $\mu_{0}:=\frac{n-\alpha}{2}$, we infer from the assumption $(\mathbf{f_{3}})$ on $f(x,u)$, the integral equations \eqref{IEe}, Lemma \ref{G-e} and \eqref{2-50-e} that, for any $2\sigma\leq|x|<+\infty$,
\begin{eqnarray}\label{2-51-e}
% \nonumber to remove numbering (before each equation)
  u(x)&\geq&\overline{C}\int_{\mathcal{C}\cap\{2|x|\leq|y|\leq4|x|\}}G_{\alpha}(x,y)|y|^{\tau}\frac{1}{|y|^{p\mu_{0}}}dy \\
  \nonumber &\geq&C\int_{\mathcal{C}\cap\{2|x|\leq|y|\leq4|x|\}}\frac{1}{|x-y|^{n-\alpha}}\cdot\frac{1}{|y|^{p\mu_{0}-\tau}}dy \\
  \nonumber &\geq&\frac{C}{|x|^{n-\alpha}}\int^{4|x|}_{2|x|}r^{n-1-p\mu_{0}+\tau}dr \\
  \nonumber &\geq&\frac{C_{1}}{|x|^{p\mu_{0}-(\alpha+\tau)}}.
\end{eqnarray}
Now, let $\mu_{1}:=p\mu_{0}-(\alpha+\tau)$. Due to $0<p<p_{s}(\tau):=\frac{n+\alpha+2\tau}{n-\alpha}$, our important observation is
\begin{equation}\label{2-52-e}
  \mu_{1}:=p\mu_{0}-(\alpha+\tau)<\mu_{0}.
\end{equation}
Thus we have obtained a better lower bound estimate than \eqref{2-50-e} after one iteration, that is,
\begin{equation}\label{2-53-e}
  u(x)\geq\frac{C_{1}}{|x|^{\mu_{1}}}, \quad\quad \forall \,\, 2\sigma\leq|x|<+\infty.
\end{equation}

For $k=0,1,2,\cdots$, define
\begin{equation}\label{2-54-e}
  \mu_{k+1}:=p\mu_{k}-(\alpha+\tau).
\end{equation}
Since $0<p<p_{s}(\tau):=\frac{n+\alpha+2\tau}{n-\alpha}$, it is easy to see that the sequence $\{\mu_{k}\}$ is monotone decreasing with respect to $k$. Continuing the above iteration process involving the integral equation \eqref{IEe}, we have the following lower bound estimates for every $k=0,1,2,\cdots$,
\begin{equation}\label{2-55-e}
  u(x)\geq\frac{C_{k}}{|x|^{\mu_{k}}}, \quad\quad \forall \,\, 2\sigma\leq|x|<+\infty.
\end{equation}
Now Theorem \ref{lower-e} follows easily from the obvious properties that as $k\rightarrow+\infty$,
\begin{equation}\label{2-56-e}
   \mu_{k}\rightarrow-\frac{\tau+\alpha}{1-p} \quad \text{if} \,\, 0<p<1;
  \quad\quad \mu_{k}\rightarrow-\infty \quad \text{if} \,\, 1\leq p<\frac{n+\alpha+2\tau}{n-\alpha}.
\end{equation}
This finishes our proof of Theorem \ref{lower-e}.
\end{proof}

One can easily observe that the lower bound estimates in Theorem \ref{lower-e} contradicts with the following integrability
\begin{equation}\label{2-57-e}
  C\int_{\mathcal{C}\cap\Omega_{2\sigma}}\frac{|x|^{a}u^{p}(x)}{|2\sigma e_{n}-x|^{n-\alpha}}dx\leq u(2\sigma e_{n})<+\infty
\end{equation}
indicated by the integral equations \eqref{IEe}, where the unit vector $e_{n}:=(0,\cdots,0,1)$. Therefore, we must have $u\equiv0$ in $\overline{\Omega_{1}}$, that is, the unique nonnegative solution to IEs \eqref{IEe} is $u\equiv0$ in $\overline{\Omega_{1}}$. This finishes our proof of Theorem \ref{Thm0} in the noncritical cases $0<\alpha<n$.

\subsection{The critical order case $\alpha=n=2$}

In the critical order case, the Green's function in the equivalent integral equations \eqref{IEe} takes the form:
\begin{eqnarray}\label{GREENe-0c}
  G_{2}(x,y)&:=&C\int_{0}^{\frac{(|x|^{2}-1)(|y|^{2}-1)}{|x-y|^{2}}}\frac{1}{1+b}db=C\ln\left[1+\frac{(|x|^{2}-1)(|y|^{2}-1)}{|x-y|^{2}}\right] \\
 \nonumber &=&2C\left(\ln\frac{1}{|x-y|}-\ln\frac{1}{\left|\frac{x}{|x|}-|x|y\right|}\right)\quad\quad \text{if} \,\,\, x,y\in \Omega_{1},
\end{eqnarray}
and $G_{2}(x,y):=0$ if $x$ or $y\in\mathbb{R}^{2}\setminus \Omega_{1}$.

Next, we will apply the method of scaling spheres to show the following lower bound estimate of positive solution $u$ for $|x|$ large, which will contradict with the integral equations \eqref{IEe}.
\begin{thm}\label{lower-ec}
Assume $\alpha=n=2$, $f(x,u)$ is subcritical and satisfies assumptions $(\mathbf{f_{1}})$, $(\mathbf{f_{2}})$ and $(\mathbf{f_{3}})$. Suppose $u$ is a positive solution to integral equations \eqref{IEe}, then it satisfies the following lower bound estimate:
\begin{equation}\label{lb1-ec}
  \inf_{|x|\geq\sigma}u(x)\geq C_{0}>0.
\end{equation}
\end{thm}
\begin{proof}
Given any $\lambda>1$, we define the Kelvin transform of $u$ centered at $0$ by
\begin{equation}\label{Kelvin-ec}
  u_{\lambda}(x):=u\left(\frac{\lambda^{2}x}{|x|^{2}}\right)
\end{equation}
for arbitrary $x\in\{x\in\overline{\Omega_{1}}\,|\,1\leq|x|\leq\lambda^{2}\}$.

Now, we will carry out the process of scaling spheres in $\Omega_{1}$ with respect to the origin $0\in\mathbb{R}^{n}$.

Let $\lambda>1$ be an arbitrary real number and let $\omega^{\lambda}(x):=u_{\lambda}(x)-u(x)$ for any $x\in B_{\lambda^{2}}(0)\setminus\overline{B_{\lambda}(0)}$. We will first show that, for $\lambda>1$ sufficiently close to $1$,
\begin{equation}\label{2-7-ec}
  \omega^{\lambda}(x)\leq0, \,\,\,\,\,\, \forall \,\, x\in B_{\lambda^{2}}(0)\setminus\overline{B_{\lambda}(0)}.
\end{equation}
Then, we start dilating the circle $S_{\lambda}:=\{x\in\mathbb{R}^{2}\,|\,|x|=\lambda\}$ from near the unit circle $S_{1}$ outward as long as \eqref{2-7-ec} holds, until its limiting position $\lambda=+\infty$ and derive lower bound estimates of $u$ for $|x|$ large. Therefore, the scaling sphere process can be divided into two steps.

\emph{Step 1. Start dilating the circle $S_{\lambda}$ from near $\lambda=1$.} Define
\begin{equation}\label{2-8-ec}
  (B_{\lambda^{2}}\setminus\overline{B_{\lambda}})^{+}:=\{x\in B_{\lambda^{2}}(0)\setminus\overline{B_{\lambda}(0)} \, | \, \omega^{\lambda}(x)>0\}.
\end{equation}
We will show that, for $\lambda>1$ sufficiently close to $1$,
\begin{equation}\label{2-9-ec}
  (B_{\lambda^{2}}\setminus\overline{B_{\lambda}})^{+}=\emptyset.
\end{equation}

Since $u$ is a positive solution to integral equations \eqref{IEe}, through direct calculations, we get, for any $\lambda>1$,
\begin{equation}\label{2-34-ec}
  u(x)=\int_{|y|>\lambda}G_{2}(x,y)f(y,u(y))dy+\int_{B_{\lambda^{2}}(0)\setminus\overline{B_{\lambda}(0)}}G_{2}(x,y^{\lambda})
  \left(\frac{\lambda}{|y|}\right)^{4}f(y^{\lambda},u_{\lambda}(y))dy
\end{equation}
for any $x\in\overline{\Omega_{1}}$. By direct calculations, one can also verify that $u_{\lambda}$ satisfies the following integral equation
\begin{equation}\label{2-35-ec}
  u_{\lambda}(x)=\int_{|y|>1}G_{2}(x^{\lambda},y)f(y,u(y))dy
\end{equation}
for any $x\in\{x\in\overline{\Omega_{1}}\,|\,1\leq|x|\leq\lambda^{2}\}$, and hence, it follows immediately that
\begin{eqnarray}\label{2-36-ec}
  u_{\lambda}(x)&=&\int_{|y|>\lambda}G_{2}(x^{\lambda},y)f(y,u(y))dy \\
 \nonumber \quad\quad &&+\int_{B_{\lambda^{2}}(0)\setminus\overline{B_{\lambda}(0)}}G_{2}(x^{\lambda},y^{\lambda})
 \left(\frac{\lambda}{|y|}\right)^{4}f(y^{\lambda},u_{\lambda}(y))dy.
\end{eqnarray}
Therefore, we have, for any $x\in B_{\lambda^{2}}(0)\setminus\overline{B_{1}(0)}$,
\begin{eqnarray}\label{omega-ec}
% \nonumber to remove numbering (before each equation)
  && \omega^{\lambda}(x)=u_{\lambda}(x)-u(x)\\
  \nonumber &=&\int_{B_{\lambda^{2}}(0)\setminus\overline{B_{\lambda}(0)}}\bigg\{\left[G_{2}(x^{\lambda},y^{\lambda})
  -G_{2}(x,y^{\lambda})\right]\left(\frac{\lambda}{|y|}\right)^{4}f(y^{\lambda},u_{\lambda}(y)) \\
 \nonumber && -\left[G_{2}(x,y)-G_{2}(x^{\lambda},y)\right]f(y,u(y))\bigg\}dy+\int_{|y|>\lambda^{2}}\left[G_{2}(x^{\lambda},y)-G_{2}(x,y)\right]f(y,u(y))dy.
\end{eqnarray}

Now we need some basic properties about the Green's function $G_{2}(x,y)$. From \eqref{GREENe-0c}, one can obtain that for any $x,\,y \in B_{\lambda^{2}}(0)\setminus\overline{B_{1}(0)}$, $x\neq y$,
\begin{equation}\label{GP-ec}
G_{2}(x,y)=C\ln\left[1+\frac{(|x|^{2}-1)(|y|^{2}-1)}{|x-y|^{2}}\right]\leq C \ln{\left(1+\frac{\lambda^4}{|x-y|^2}\right)}.
\end{equation}
It is well known that
\begin{equation}\label{wk-ec}
\ln{(1+t)}=o(t^\varepsilon), \,\,\quad \text{as}\,\, t\rightarrow +\infty,
\end{equation}
where $\varepsilon$ is an arbitrary positive real number. This implies, for any given $\varepsilon>0$, there exists a $\delta(\varepsilon)>0$ such that
\begin{equation}\label{ln-ec}
\ln{(1+t)}\leq t^\varepsilon, \,\,\qquad \forall \, t>\frac{1}{{\delta(\varepsilon)}^2}.
\end{equation}

Therefore, by \eqref{GP-ec}, \eqref{ln-ec} and straightforward calculations, we have the following Lemma on properties of the Green's function $G_{2}(x,y)$.
\begin{lem}\label{G-ec}
The Green's function $G_{2}(x,y)$ satisfies the following point-wise estimates:
\begin{flalign}
\nonumber &\text{$(i)\,\, G_{2}(x,y)\leq C\lambda^{4\varepsilon} \frac{1}{|x-y|^{2\varepsilon}}, \,\,\qquad \forall\,\,\,1<|x|,\,|y|<\lambda^{2}, \, |x-y|<\lambda^{2}\delta(\varepsilon);$}& \\
\nonumber &\text{$(ii) \,\, G_{2}(x,y)\leq C \ln\left(1+\frac{1}{{\delta(\varepsilon)}^2}\right), \,\,\qquad \forall\,\,\,1<|x|,\,|y|<\lambda^{2}, \,
|x-y|\geq\lambda^{2}\delta(\varepsilon);$}& \\
\nonumber &\text{$(iii) \,\, G_{2}(x,y)\geq C>0, \,\,\qquad \forall\,\,\,|x|,\,|y|\geq2;$}& \\
\nonumber &\text{$(iv) \,\, G_{2}(x^{\lambda},y)-G_{2}(x,y)\leq0, \quad\quad \forall \,\, \lambda<|x|<\lambda^{2}, \,\, \lambda<|y|<+\infty;$}& \\
\nonumber &\text{$(v) \,\, G_{2}(x^{\lambda},y^{\lambda})-G_{2}(x,y^{\lambda})\leq G_{2}(x,y)-G_{2}(x^{\lambda},y), \qquad \forall \,\, \lambda<|x|,|y|<\lambda^{2}.$}&
\end{flalign}
\end{lem}

Lemma \ref{G-ec} can be proved by direct calculations, so we omit the details here.

From Lemma \ref{G-ec} and \eqref{omega-ec}, one can derive that, for any $x\in B_{\lambda^{2}}(0)\setminus\overline{B_{\lambda}(0)}$,
\begin{eqnarray}\label{2-37-ec}
% \nonumber to remove numbering (before each equation)
  &&\omega^{\lambda}(x)=u_{\lambda}(x)-u(x) \\
 \nonumber &\leq&\int_{\lambda<|y|<\lambda^{2}}\left[G_{2}(x,y)-G_{2}(x^{\lambda},y)\right]
 \left[\left(\frac{\lambda}{|y|}\right)^{4}f(y^{\lambda},u_{\lambda}(y))-f(y,u(y))\right]dy\\
\nonumber &<&\int_{\lambda<|y|<\lambda^{2}}\left(G_{2}(x,y)-G_{2}(x^{\lambda},y)\right)\left(f(y,u_{\lambda}(y))-f(y,u(y))\right)dy\\
\nonumber &\leq&C\int_{\left(B_{\lambda^{2}}\setminus\overline{B_{\lambda}}\right)^{+}}G_{2}(x,y)\frac{f(y,u_{\lambda}(y))-f(y,u(y))}{u_{\lambda}(y)-u(y)}
\omega^{\lambda}(y)dy\\
\nonumber &\leq&C\lambda^{4\varepsilon}\int_{\left(B_{\lambda^{2}}\setminus\overline{B_{\lambda}}\right)^{+}\cap B_{\lambda^{2}\delta(\varepsilon)}(x)}\frac{1}{|x-y|^{2\varepsilon}}\cdot\frac{f(y,u_{\lambda}(y))-f(y,u(y))}{u_{\lambda}(y)-u(y)}
\omega^{\lambda}(y)dy\\
\nonumber &&+C(\delta(\varepsilon))\int_{\left(B_{\lambda^{2}}\setminus\overline{B_{\lambda}}\right)^{+}\setminus B_{\lambda^{2}\delta(\varepsilon)}(x)}\frac{f(y,u_{\lambda}(y))-f(y,u(y))}{u_{\lambda}(y)-u(y)}
\omega^{\lambda}(y)dy,
\end{eqnarray}
where we have used the subcritical condition on $f(x,u)$ for $\mu=\left(\frac{\lambda}{|y|}\right)^{2}<1$ to derive the second inequality and the assumption $(\mathbf{f_{1}})$ on $f(x,u)$ to derive the third inequality.

By Hardy-Littlewood-Sobolev inequality, H\"{o}lder inequality and \eqref{2-37-ec}, we have, for any $\frac{1}{\varepsilon}<q<+\infty$,
\begin{eqnarray}\label{3-14-ec}
 && \|\omega^{\lambda}\|_{L^{q}((B_{\lambda^{2}}\setminus\overline{B_{\lambda}})^{+})}\\
 \nonumber &\leq& C\lambda^{4\varepsilon}\left\|\frac{f(y,u_{\lambda}(y))-f(y,u(y))}{u_{\lambda}(y)-u(y)}
\omega^{\lambda}(y)\right\|_{L^{\frac{nq}{n+(n-2\varepsilon)q}}((B_{\lambda^{2}}\setminus\overline{B_{\lambda}})^{+})}\\
\nonumber &&+C(\delta(\varepsilon))\,\left|\left(B_{\lambda^{2}}\setminus\overline{B_{\lambda}}\right)^{+}\right|^{\frac{1}{q}}
\int_{\left(B_{\lambda^{2}}\setminus\overline{B_{\lambda}}\right)^{+}}\left|\frac{f(y,u_{\lambda}(y))-f(y,u(y))}{u_{\lambda}(y)-u(y)}\right|\omega^{\lambda}(y)dy\\
  \nonumber &\leq& C\lambda^{4\varepsilon}
  \left\|\frac{f(y,u_{\lambda}(y))-f(y,u(y))}{u_{\lambda}(y)-u(y)}\right\|_{L^{\frac{n}{n-2\varepsilon}}((B_{\lambda^{2}}\setminus\overline{B_{\lambda}})^{+})}
  \|\omega^{\lambda}\|_{L^{q}((B_{\lambda^{2}}\setminus\overline{B_{\lambda}})^{+})}\\
  \nonumber &&+C(\delta(\varepsilon))\,\left|\left(B_{\lambda^{2}}\setminus\overline{B_{\lambda}}\right)^{+}\right|^{\frac{1}{q}}
  \left\|\frac{f(y,u_{\lambda}(y))-f(y,u(y))}{u_{\lambda}(y)-u(y)}\right\|_{L^{\frac{q}{q-1}}((B_{\lambda^{2}}\setminus\overline{B_{\lambda}})^{+})}
  \|\omega^{\lambda}\|_{L^{q}((B_{\lambda^{2}}\setminus\overline{B_{\lambda}})^{+})}.
\end{eqnarray}
Since $\theta<1$ in the assumption $(\mathbf{f_{2}})$, we first choose $\varepsilon>0$ sufficiently small such that $-\frac{n\theta}{n-2\varepsilon}>-1$, then choose $q>\frac{1}{\varepsilon}$ sufficiently large such that $-\frac{q\theta}{q-1}>-1$. Then, since $u\in C(\overline{\Omega_{1}})$ and $f(x,u)$ satisfies the assumption $(\mathbf{f_{2}})$, there exists a $\delta_{0}>0$ small enough, such that
\begin{eqnarray}\label{3-15-ec}
 && C\lambda^{4\varepsilon}
  \left\|\frac{f(y,u_{\lambda}(y))-f(y,u(y))}{u_{\lambda}(y)-u(y)}\right\|_{L^{\frac{n}{n-2\varepsilon}}((B_{\lambda^{2}}\setminus\overline{B_{\lambda}})^{+})}\\
 \nonumber &&\qquad+C(\delta(\varepsilon))\,\left|\left(B_{\lambda^{2}}\setminus\overline{B_{\lambda}}\right)^{+}\right|^{\frac{1}{q}}
  \left\|\frac{f(y,u_{\lambda}(y))-f(y,u(y))}{u_{\lambda}(y)-u(y)}\right\|_{L^{\frac{q}{q-1}}((B_{\lambda^{2}}\setminus\overline{B_{\lambda}})^{+})}\leq\frac{1}{2}
\end{eqnarray}
for all $1<\lambda\leq1+\delta_{0}$, and hence \eqref{3-14-ec} implies
\begin{equation}\label{3-16-ec}
  \|\omega^{\lambda}\|_{L^{q}((B_{\lambda^{2}}\setminus\overline{B_{\lambda}})^{+})}=0,
\end{equation}
which means $(B_{\lambda^{2}}\setminus\overline{B_{\lambda}})^{+}=\emptyset$. Therefore, we have proved for all $1<\lambda\leq1+\delta_{0}$, $(B_{\lambda^{2}}\setminus\overline{B_{\lambda}})^{+}=\emptyset$, that is,
\begin{equation}\label{3-17-ec}
  \omega^{\lambda}(x)\leq0, \,\,\,\,\,\,\, \forall \, x\in B_{\lambda^{2}}(0)\setminus\overline{B_{\lambda}(0)}.
\end{equation}
This completes Step 1.

\emph{Step 2. Dilate the circle $S_{\lambda}$ outward until $\lambda=+\infty$ to derive lower bound estimates of $u$ for $|x|$ large.} Step 1 provides us a start point to dilate the circle $S_{\lambda}$ from near $\lambda=1$. Now we dilate the circle $S_{\lambda}$ outward as long as \eqref{2-7-ec} holds. Let
\begin{equation}\label{2-29-ec}
  \lambda_{0}:=\sup\{1<\lambda<+\infty\,|\, \omega^{\mu}\leq0 \,\, in \,\, B_{\mu^{2}}(0)\setminus\overline{B_{\mu}(0)}, \,\, \forall \, 1<\mu\leq\lambda\}\in(1,+\infty],
\end{equation}
and hence, one has
\begin{equation}\label{2-30-ec}
  \omega^{\lambda_{0}}(x)\leq0, \quad\quad \forall \,\, x\in B_{\lambda_{0}^{2}}(0)\setminus\overline{B_{\lambda_{0}}(0)}.
\end{equation}
In what follows, we will prove $\lambda_{0}=+\infty$ by contradiction arguments.

Suppose on contrary that $1<\lambda_{0}<+\infty$. In order to get a contradiction, we will first prove
\begin{equation}\label{2-31-ec}
  \omega^{\lambda_{0}}(x)\equiv0, \,\,\,\,\,\,\forall \, x\in B_{\lambda_{0}^{2}}(0)\setminus\overline{B_{\lambda_{0}}(0)}
\end{equation}
by using contradiction arguments.

Suppose on contrary that \eqref{2-31-ec} does not hold, that is, $\omega^{\lambda_{0}}\leq0$ but $\omega^{\lambda_{0}}$ is not identically zero in $B_{\lambda_{0}^{2}}(0)\setminus\overline{B_{\lambda_{0}}(0)}$, then there exists a $x^{0}\in B_{\lambda_{0}^{2}}(0)\setminus\overline{B_{\lambda_{0}}(0)}$ such that $\omega^{\lambda_{0}}(x^{0})<0$. We will obtain a contradiction with \eqref{2-29-ec} via showing that the circle $S_{\lambda}$ can be dilated outward a little bit further, more precisely, there exists a $\epsilon>0$ small enough such that $\omega^{\lambda}\leq0$ in $B_{\lambda^{2}}(0)\setminus\overline{B_{\lambda}(0)}$ for all $\lambda\in[\lambda_{0},\lambda_{0}+\epsilon]$.

For that purpose, we will first show that
\begin{equation}\label{2-32-ec}
  \omega^{\lambda_{0}}(x)<0, \,\,\,\,\,\, \forall \, x\in B_{\lambda_{0}^{2}}(0)\setminus\overline{B_{\lambda_{0}}(0)}.
\end{equation}
Indeed, since we have assumed there exists a point $x^{0}\in B_{\lambda_{0}^{2}}(0)\setminus\overline{B_{\lambda_{0}}(0)}$ such that $\omega^{\lambda_{0}}(x^{0})<0$, by continuity, there exists a small $\delta>0$ and a constant $c_{0}>0$ such that
\begin{equation}\label{2-33-ec}
B_{\delta}(x^{0})\subset B_{\lambda_{0}^{2}}(0)\setminus\overline{B_{\lambda_{0}}(0)} \,\,\,\,\,\, \text{and} \,\,\,\,\,\,
\omega^{\lambda_{0}}(x)\leq -c_{0}<0, \,\,\,\,\,\,\,\, \forall \, x\in B_{\delta}(x^{0}).
\end{equation}
Since $f(x,u)$ is subcritical and satisfies the assumption $(\mathbf{f_{1}})$, one can derive from \eqref{2-33-ec}, Lemma \ref{G-ec} and \eqref{2-37-ec} that, for any $x\in B_{\lambda_{0}^{2}}(0)\setminus\overline{B_{\lambda_{0}}(0)}$,
\begin{eqnarray}\label{9-37-ec}
% \nonumber to remove numbering (before each equation)
  &&\omega^{\lambda_{0}}(x)=u_{\lambda_{0}}(x)-u(x) \\
  \nonumber &\leq&\int_{\lambda_{0}<|y|<\lambda_{0}^{2}}\left[G_{2}(x,y)-G_{2}(x^{\lambda_{0}},y)\right] \left[\left(\frac{\lambda_{0}}{|y|}\right)^{4}f(y^{\lambda_{0}},u_{\lambda_{0}}(y))-f(y,u(y))\right]dy\\
\nonumber &<&\int_{B_{\delta}(x_{0})}\left(G_{2}(x,y)-G_{2}(x^{\lambda_{0}},y)\right)\left(f(y,u_{\lambda_{0}}(y))-f(y,u(y))\right)dy\leq0,
\end{eqnarray}
thus we arrive at \eqref{2-32-ec}.

Now, we choose a $0<r_{0}<\frac{1}{4}\min\{\lambda_{0}^{2}-\lambda_{0},1\}$ small enough, such that
\begin{eqnarray}\label{9-0ec}
 && C\lambda^{4\varepsilon}
  \left\|\frac{f(y,u_{\lambda}(y))-f(y,u(y))}{u_{\lambda}(y)-u(y)}\right\|_{L^{\frac{n}{n-2\varepsilon}}\big(A_{\lambda_{0}+r_{0},r_{0}}\cup A_{\lambda_{0}^{2}+r_{0},2r_{0}}\big)}\\
 \nonumber &&\quad+C(\delta(\varepsilon))\,\left|\left(B_{\lambda^{2}}\setminus\overline{B_{\lambda}}\right)^{+}\right|^{\frac{1}{q}}
  \left\|\frac{f(y,u_{\lambda}(y))-f(y,u(y))}{u_{\lambda}(y)-u(y)}\right\|_{L^{\frac{q}{q-1}}\big(A_{\lambda_{0}+r_{0},r_{0}}\cup A_{\lambda_{0}^{2}+r_{0},2r_{0}}\big)}\leq\frac{1}{2}
\end{eqnarray}
for any $\lambda\in[\lambda_{0},\lambda_{0}+\frac{r_{0}}{4}]$, where the choices of $\varepsilon$, $q$ and the constants $C$, $C(\delta(\varepsilon))$ are the same as in \eqref{3-15-ec}. By \eqref{2-37-ec}, one can easily verify that inequality as \eqref{3-14-ec} (with the same constants $C$ and $C(\delta(\varepsilon))$) also holds for any $\lambda\in[\lambda_{0},\lambda_{0}+\frac{r_{0}}{4}]$, that is, for any $\frac{1}{\varepsilon}<q<+\infty$,
\begin{eqnarray}\label{9-2ec}
 && \|\omega^{\lambda}\|_{L^{q}((B_{\lambda^{2}}\setminus\overline{B_{\lambda}})^{+})}\\
 \nonumber &\leq& C\lambda^{4\varepsilon}
  \left\|\frac{f(y,u_{\lambda}(y))-f(y,u(y))}{u_{\lambda}(y)-u(y)}\right\|_{L^{\frac{n}{n-2\varepsilon}}((B_{\lambda^{2}}\setminus\overline{B_{\lambda}})^{+})}
  \|\omega^{\lambda}\|_{L^{q}((B_{\lambda^{2}}\setminus\overline{B_{\lambda}})^{+})}\\
  \nonumber &&+C(\delta(\varepsilon))\,\left|\left(B_{\lambda^{2}}\setminus\overline{B_{\lambda}}\right)^{+}\right|^{\frac{1}{q}}
  \left\|\frac{f(y,u_{\lambda}(y))-f(y,u(y))}{u_{\lambda}(y)-u(y)}\right\|_{L^{\frac{q}{q-1}}((B_{\lambda^{2}}\setminus\overline{B_{\lambda}})^{+})}
  \|\omega^{\lambda}\|_{L^{q}((B_{\lambda^{2}}\setminus\overline{B_{\lambda}})^{+})}.
\end{eqnarray}

By \eqref{2-32-ec}, we can define
\begin{equation}\label{2-40-ec}
  M_{1}:=\sup_{x\in \overline{B_{\lambda_{0}^{2}-r_{0}}(0)}\setminus B_{\lambda_{0}+r_{0}}(0)}\omega^{\lambda_{0}}(x)<0.
\end{equation}
Since $u$ is uniformly continuous on arbitrary compact set $K\subset\overline{\Omega_{1}}$, we can deduce from \eqref{2-40-ec} that, there exists a $0<\epsilon_{1}<\frac{r_{0}}{4}$ sufficiently small, such that, for any $\lambda\in[\lambda_{0},\lambda_{0}+\epsilon_{1}]$,
\begin{equation}\label{2-41-ec}
  \omega^{\lambda}(x)\leq\frac{M_{1}}{2}<0, \,\,\,\,\,\, \forall \, x\in \overline{B_{\lambda_{0}^{2}-r_{0}}(0)}\setminus B_{\lambda_{0}+r_{0}}(0).
\end{equation}

For any $\lambda\in[\lambda_{0},\lambda_{0}+\epsilon_{1}]$, it follows from \eqref{2-41-ec} that
\begin{equation}\label{9-4ec}
  (B_{\lambda^{2}}\setminus\overline{B_{\lambda}})^{+}\subset A_{\lambda_{0}+r_{0},r_{0}}\cup A_{\lambda_{0}^{2}+r_{0},2r_{0}}.
\end{equation}
As a consequence of \eqref{9-0ec}, \eqref{9-2ec} and \eqref{9-4ec}, we get
\begin{equation}\label{9-5ec}
  \|\omega^{\lambda}\|_{L^{q}\left((B_{\lambda^{2}}\setminus\overline{B_{\lambda}})^{+}\right)}=0,
\end{equation}
and hence $(B_{\lambda^{2}}\setminus\overline{B_{\lambda}})^{+}=\emptyset$ for all $\lambda\in[\lambda_{0},\lambda_{0}+\epsilon_{1}]$, that is,
\begin{equation}\label{2-45-ec}
  \omega^{\lambda}(x)\leq0, \,\,\,\,\,\,\, \forall \,\, x\in B_{\lambda^{2}}(0)\setminus\overline{B_{\lambda}(0)},
\end{equation}
which contradicts with the definition \eqref{2-29-ec} of $\lambda_{0}$. As a consequence, in the case $1<\lambda_{0}<+\infty$, \eqref{2-31-ec} must hold true, that is,
\begin{equation}\label{2-46-ec}
  \omega^{\lambda_{0}}\equiv0 \,\,\,\,\,\, \text{in} \,\,\, B_{\lambda_{0}^{2}}(0)\setminus\overline{B_{\lambda_{0}}(0)}.
\end{equation}

However, by subcritical condition on $f(x,u)$, the first inequality in \eqref{9-37-ec} and \eqref{2-46-ec}, we arrive at
\begin{eqnarray}\label{2-47-ec}
 && 0=\omega^{\lambda_{0}}(x)=u_{\lambda_{0}}(x)-u(x)\\
 \nonumber &\leq&\int_{\lambda_{0}<|y|<\lambda_{0}^{2}}\left[G_{2}(x,y)-G_{2}(x^{\lambda_{0}},y)\right] \left[\left(\frac{\lambda_{0}}{|y|}\right)^{4}f(y^{\lambda_{0}},u_{\lambda_{0}}(y))-f(y,u(y))\right]dy\\
\nonumber &<&\int_{\lambda_{0}<|y|<\lambda_{0}^{2}}\left(G_{2}(x,y)-G_{2}(x^{\lambda_{0}},y)\right)\left(f(y,u_{\lambda_{0}}(y))-f(y,u(y))\right)dy=0
\end{eqnarray}
for any $x\in B_{\lambda_{0}^{2}}(0)\setminus\overline{B_{\lambda_{0}}(0)}$, which is absurd. Thus we must have $\lambda_{0}=+\infty$, that is,
\begin{equation}\label{2-48-ec}
  u(x)\geq u\left(\frac{\lambda^{2}x}{|x|^{2}}\right), \quad\quad \forall \,\, \lambda\leq|x|\leq\lambda^{2}, \quad \forall \,\, 1<\lambda<+\infty.
\end{equation}
For arbitrary $\sigma\leq|x|<+\infty$, let $1<\sqrt{|x|}<\lambda:=\sqrt{\sigma|x|}\leq|x|<+\infty$, then \eqref{2-48-ec} yields that
\begin{equation}\label{2-49-ec}
  u(x)\geq u\left(\frac{\sigma x}{|x|}\right),
\end{equation}
and hence, we arrive at the following lower bound estimate of $u$:
\begin{equation}\label{2-50-ec}
  u(x)\geq\min_{x\in S_{\sigma}}u(x):=C_{0}>0, \quad\quad \forall \,\, \sigma\leq|x|<\infty.
\end{equation}

This finishes our proof of Theorem \ref{lower-ec}.
\end{proof}

Since $-2\leq\tau<+\infty$ in assumption $(\mathbf{f_{3}})$, we can deduce from the assumption $(\mathbf{f_{3}})$ on $f(x,u)$, the integral equations \eqref{IEe}, Lemma \ref{G-ec} and Theorem \ref{lower-ec} that, for any $2\sigma\leq|x|<+\infty$,
\begin{eqnarray}\label{2-51-ec'}
% \nonumber to remove numbering (before each equation)
  u(x)&\geq&\overline{C}\int_{\mathcal{C}\cap\{|y|\geq2|x|\}}G_{2}(x,y)|y|^{\tau}C^{p}_{0}dy \\
  \nonumber &\geq&C\int_{\mathcal{C}\cap\{|y|\geq2|x|\}}|y|^{\tau}dy=+\infty,
\end{eqnarray}
which is a contradiction! Therefore, we must have $u\equiv0$ in $\overline{\Omega_{1}}$, that is, the unique nonnegative solution to IEs \eqref{IEe} with $\alpha=n=2$ is $u\equiv0$ in $\overline{\Omega_{1}}$.

This concludes our proof of Theorem \ref{Thm0}.

\end{document}